\documentclass[reqno,oneside,11pt]{amsart}

\usepackage[a4paper, hmargin={2.8cm, 2.8cm}, vmargin={2.5cm, 2.5cm}]{geometry}  
\usepackage[danish,english]{babel} 
\usepackage[latin1]{inputenc} 
\usepackage[T1]{fontenc} 
\usepackage{lmodern}
\usepackage{amscd}
\usepackage{amsmath,amsthm,amssymb,amsfonts,mathrsfs,latexsym} 
\usepackage{graphicx} 
\usepackage{verbatim} 
\usepackage[all]{xy} 
\usepackage[pagebackref, colorlinks, linkcolor=red, citecolor=blue, urlcolor=blue, hypertexnames=true]{hyperref}  
\usepackage{multirow}
\usepackage{xcolor}
\usepackage{enumerate}

\makeatletter
\newtheorem*{rep@theorem}{\rep@title}
\newcommand{\newreptheorem}[2]{%
\newenvironment{rep#1}[1]{%
 \def\rep@title{#2 \ref{##1}}%
 \begin{rep@theorem}}%
 {\end{rep@theorem}}}
\makeatother

\newtheorem{thm}{Theorem}[section]
\newtheorem{thmx}{Theorem}
\newtheorem{corx}[thmx]{Corollary}
\newreptheorem{thm}{Theorem}

\newtheorem{prop}[thm]{Proposition}
\newtheorem{cor}[thm]{Corollary}
\newtheorem*{thm*}{Theorem}
\newtheorem*{problem*}{Problem}

\newtheorem*{claim*}{Claim}
\theoremstyle{definition}
\newtheorem{defi}[thm]{Definition}

\newtheorem{rem}[thm]{Remark}

\newcommand{\claimmark}{\hfill$\lozenge$}

\newcommand{\Aut}{\operatorname{Aut}}
\newcommand{\End}{\operatorname{End}}
\renewcommand{\top}{\operatorname{top}}
\newcommand{\Prim}{\operatorname{Prim}}
\newcommand{\Ind}{\operatorname{Ind}}

\newcommand{\C}{\mathbb{C}}
\newcommand{\N}{\mathbb{N}}

\newcommand{\Z}{\mathbb{Z}}

\newcommand{\ra}{\rightarrow}
\renewcommand{\epsilon}{\varepsilon}
\renewcommand{\phi}{\varphi}
\renewcommand{\tilde}{\widetilde}
\renewcommand{\hat}{\widehat}

\DeclareMathOperator{\HH}{H}
\DeclareMathOperator{\GL}{GL}
\DeclareMathOperator{\SL}{SL}

\DeclareMathOperator{\Sp}{Sp}
\DeclareMathOperator{\Mp}{Mp}

\renewcommand{\det}{\operatorname{det}}

\DeclareFontFamily{U}{mathx}{\hyphenchar\font45}
\DeclareFontShape{U}{mathx}{m}{n}{
      <5> <6> <7> <8> <9> <10>
      <10.95> <12> <14.4> <17.28> <20.74> <24.88>
      mathx10
      }{}
\DeclareSymbolFont{mathx}{U}{mathx}{m}{n}
\DeclareFontSubstitution{U}{mathx}{m}{n}
\DeclareMathAccent{\widecheck}{0}{mathx}{"71}
\DeclareMathAccent{\wideparen}{0}{mathx}{"75}

\newcommand*{\blue}[1]{{\color{blue}#1}}

\numberwithin{equation}{section}

\begin{document}
\selectlanguage{english} 


\title[]{\texorpdfstring{The orbit method for the Baum-Connes Conjecture for algebraic groups over local function fields}{The orbit method for the Baum-Connes Conjecture for algebraic groups over local function fields}}


\author{Siegfried Echterhoff$^{1}$}
\address{Mathematisches Institut der WWU M\"unster,
\newline Einsteinstrasse 62, 48149 M\"unster, Germany}
\email{echters@uni-muenster.de}

\author{Kang Li$^{2}$}
\address{Mathematisches Institut der WWU M\"unster,
\newline Einsteinstrasse 62, 48149 M\"unster, Germany}
\email{lik@uni-muenster.de}

\author{Ryszard Nest$^{3}$}
\address{Department of Mathematical Sciences, University of Copenhagen,
\newline Universitetsparken 5, DK-2100 Copenhagen \O, Denmark}
\email{rnest@math.ku.dk}

\thanks{{$^{1}$} Supported by Deutsche Forschungsgemeinschaft (SFB 878, Groups, Geometry and Actions).}

\thanks{{$^{2}$} Supported by the Danish Council for Independent Research (DFF-5051-00037) and is partially supported by the DFG (SFB 878).}

\thanks{{$^{3}$} Supported by the Danish National Research Foundation through the Centre for Symmetry and Deformation (DNRF92).}

\begin{abstract}
The main purpose of this paper is to modify the orbit method for the Baum-Connes conjecture as developed by Chabert, Echterhoff and Nest 
in their proof of the  Connes-Kasparov conjecture for almost connected groups \cite{MR2010742} in order to deal with linear algebraic groups over local function fields (i.e., non-archimedean local fields of positive characteristic). As a consequence, we verify the Baum-Connes conjecture for certain Levi-decomposable linear algebraic groups over local function fields. One of these is the Jacobi group, which is the semidirect product of the symplectic group and the Heisenberg group.
\end{abstract}

\date{\today}
\maketitle
\parskip 4pt
 \section{Introduction}
 The Baum-Connes conjecture was first introduced by Paul Baum and Alain Connes in 1982 \cite{MR1769535}. However, the paper was published 18 years later and its current formulation was given in \cite{MR1292018}. The origin of the conjecture goes back to Connes' foliation theory \cite{MR679730} and Baum's geometric description of K-homology theory \cite{MR679698}.
 
Consider a locally compact second countable group $G$ and a separable $G$-$C^*$-algebra $A$. Let $\underline{E}(G)$ denote a locally compact universal proper $G$-space. Such a space always exists and is unique up to $G$-homotopy equivalence (see \cite{MR1292018} and \cite{MR1998480} for details). The topological K-theory of $G$ with coefficient $A$ is defined as
\begin{align*}
K_*^{\text{top}}(G;A):=\lim_X KK^G_*(C_0(X),A),
\end{align*} 
where $X$ runs through all $G$-invariant subspaces of $\underline{E}(G)$ such that $X/G$ is compact, and $KK^G_*(C_0(X),A)$ denotes Kasparov's equivariant KK-theory. Let $A\rtimes_rG$ denote the reduced crossed product of $A$ by $G$. Then the
Baum-Connes conjecture with coefficient $A$ states that the {\em assembly map}
$$\mu_A: K_*^{\text{top}}(G;A)\to K_*(A\rtimes_rG),$$
as constructed by Baum, Connes and Higson in \cite{MR1292018}, is an isomorphism of abelian groups. If this holds, we say that {\em $G$ satisfies
BC (the Baum-Connes conjecture) for $A$. }

The Baum-Connes conjecture has been verified for some large families of groups. In particular, Higson and Kasparov in \cite{MR1821144} proved the Baum-Connes conjecture with arbitrary coefficients for locally compact second countable groups having the Haagerup property (e.g. amenable groups). In \cite{MR1914617} Lafforgue proved the Baum-Connes conjecture for $\C$ for all reductive Lie groups whose semi-simple part has finite center and reductive linear algebraic groups over non-archimedean local fields. Later, Chabert, Echterhoff and Nest in \cite{MR2010742} used these results together with
the permanence properties of the conjecture as studied in \cite{MR1836047, MR2100669}
to verify the Connes-Kasparov conjecture, i.e., the Baum-Connes conjecture with trivial coefficients for all almost connected second countable groups.
Using similar methods, they also settled BC with trivial coefficient $\C$ for  all linear algebraic groups over local fields with zero characteristic.  However, the conjecture (even for $\C$) is still open for linear algebraic groups over local fields of positive characteristic.

A central step in the proof of the Connes-Kasparov conjecture in \cite{MR2010742} was the development of an orbit method for proving the 
Baum-Connes conjecture with coefficients built after the Mackey-Rieffel-Green orbit method for computing irreducible representations of crossed products
via the the orbit-structure for the action of $G$ on $\Prim(A)$. In \S2 we shall prove a slightly more general version of the orbit method 
for the Baum-Connes conjecture with coefficient $A$ which will work particularly well if $A$ is a type I C*-algebra.
This result, together with the permanence results for group extensions as obtained in \cite{MR1836047, MR2100669} then allows 
to formulate an orbit method for proving the Baum-Connes conjecture with trivial coefficients for a group $G$ 
which is an extension $1\to N\to G\to G/N\to 1$ such that $N$ and $G/N$ satisfy some good properties.

The idea of using the orbit method for the Baum-Connes conjecture first appeared in 
\cite {MR1836047}, where Chabert and Echterhoff proved the conjecture for $k^n\rtimes \SL_n(k)$ for an archimedean local field $k$. In order to apply the orbit method to linear algebraic $p$-adic groups we need  Kirillov's orbit method for $p$-adic unipotent groups, which is established by Moore in \cite{MR0181701}. Recently, based on Howe's early work in \cite{MR0492059}, Echterhoff and Kl\"uver obtained a version of Kirillov's orbit method for unipotent groups defined over local fields of positive characteristic $p$, provided that its nilpotence length is less than $p$ (see  \cite{MR3012147}). 
Combining this with the general results obtained in \S2 we shall prove the following results in \S3:

\begin{thmx}[see Corollary~\ref{orbit for BC}]\label{into1}
Let $k$ be a non-archimedean local field of positive characteristic $p$ and let $\mathfrak{n}$ be a finite-dimensional nilpotent $k$-Lie algebra with nilpotence length $l<p$. Let $N:=\exp(\mathfrak{n})$ be the corresponding unipotent $k$-group and let $G$ be a locally compact second countable exact group. Suppose that $G$ acts continuously on $N(k)$ such that all $G$-orbits in $\hat{N(k)}$ are locally closed.

Then the semidirect product group $N(k)\rtimes G$ satisfies BC for $\C$, if the stabilizer $G_\pi$ satisfies BC for $\mathcal{K}(H_\pi)$ for all $[\pi]\in \hat{N(k)}$.
\end{thmx}

Notice that we do not require $G$ to be a linear algebraic group in Theorem~\ref{into1}. If we only consider linear algebraic groups which admit a Levi decomposition over local function fields, we can obtain the following simplified version:
\begin{thmx}[see Theorem~\ref{algebraic orbit method}]
Let $G$ be a linear algebraic group with a Levi decomposition $G=N\rtimes R$ over a non-archimedean local field $k$ of positive characteristic $p$. Assume that $N=\exp(\mathfrak{n})$ for a finite-dimensional nilpotent $k$-Lie algebra $\mathfrak{n}$ which  has nilpotence length $l<p$.

Then $G(k)=N(k)\rtimes R(k)$ satisfies BC for $\C$, if the stabilizer $R(k)_\pi$ satisfies BC for $\mathcal{K}(H_\pi)$ for all $[\pi]\in \hat{N(k)}$.
\end{thmx}

As a consequence, we obtain the following corollary:
\begin{corx}[see Corollary~\ref{ex1} and Corollary~\ref{ex2}]\label{intocor}
 Let $k$ be a non-archimedean local field of positive characteristic $p>2$ and $n\in \N$. Then the following groups satisfy BC for $\C$:
\begin{itemize}
\item[(1)] $k^n\rtimes \GL_n(k)$, where $\GL_n(k)$ acts on $k^n$ by matrix multiplications.
\item[(2)] $k\rtimes \GL_n(k)$, where $\GL_n(k)$ acts on $k$ by $g.v=\det(g)^p \cdot v$ for $g\in \GL_n(k)$ and $v\in k$.
\item[(3)] $\HH_{2n+1}(k)\rtimes \Sp_{2n}(k)$, where $\HH_{2n+1}(k)$ denotes the $2n+1$ dimensional Heisenberg group over $k$ and  the symplectic group 
$\Sp_{2n}(k)$ acts on $k^{2n}\cong \HH_{2n+1}(k)/Z(\HH_{2n+1}(k))$ by matrix multiplications and trivially on the center $Z(\HH_{2n+1}(k))\cong k$.
\item[(4)] $k^{2n}\rtimes \Sp_{2n}(k)$, where 
$\Sp_{2n}(k)$ acts on $k^{2n} $ by matrix multiplications.
\item[(5)] $\HH_{2n+1}(k)\rtimes \Mp_{2n}(k)$, where the metaplectic group $\Mp_{2n}(k)$ acts on $\HH_{2n+1}(k)$ through the action of $\Sp_{2n}(k)$.
\item[(6)] $k^{2n}\rtimes \Mp_{2n}(k)$, where $\Mp_{2n}(k)$ acts on $k^{2n}$ through the action of $\Sp_{2n}(k)$.
\end{itemize}
\end{corx}
In Corollary~\ref{intocor} the condition $p>2$ is not necessary for (1) and (2). 

\ \newline
{\bf Acknowledgments}. The second-named author would like to thank Roger Howe, Vincent Lafforgue, George McNinch and Maarten Solleveld for helpful and enlightening discussions on a wide variety of topics.

\section{The orbit method for the Baum-Connes conjecture with coefficients}

In this section we want to discuss a general orbit method inspired by Mackey's theory of induced representations 
in order to prove the Baum-Connes conjecture for type I coefficients $A$ depending on the conjecture for the stabilisers 
$G_\pi$  for the action of $G$ on the space $\widehat{A}$ of equivalence classes of irreducible representations of $A$.
The results of this section slightly extend similar results obtained in \cite{MR2010742}.
We need to recall a classical theorem of  Glimm. Recall that a (not necessarily Hausdorff)  topological space is called 
   \begin{itemize}
\item \emph{locally compact} if every point has a neighborhood basis consisting of compact sets, and
\item \emph{almost Hausdorff} if every non-empty closed subset contains a non-empty relatively open Hausdorff subset.
\end{itemize}
Moreover, a subset $Y\subseteq X$ is called \emph{locally closed in $X$} if $Y$ is relatively open in its closure $\bar{Y}\subseteq X$. 
Equivalently, $Y$ is  an intersection of an open set and a closed set in $X$.

It is well-known that every locally closed subset of a locally compact space is locally compact. Conversely, every locally compact subset of an almost Hausdorff locally compact space is necessarily locally closed.
\begin{thm}[Glimm's Theorem, see {\cite[Theorem~1]{MR0136681}} and {\cite[Remark~3.5]{MR2010742}}]\label{glimm}
Let $G$ be a locally compact second countable group and let $X$ be a locally compact second countable almost Hausdorff $G$-space. Then the following are equivalent:
\begin{itemize}
\item[(1)] Each orbit $G\cdot x$ is locally closed.
\item[(2)] For each $x\in X$, the map $gG_x\mapsto g\cdot x$ is a homeomorphism from $G/G_x$ onto $G\cdot x$.
\item[(3)] The orbit space $X/G$ is almost Hausdorff.
\item[(4)] There exists a sequence of $G$-invariant open subsets $\{U_\nu\}_\nu$ of $X$, where $\nu$ runs through the ordinal numbers such that
\begin{itemize}
\item[(a)] $U_\nu\subseteq U_{\nu+1}$ for each $\nu$ and $(U_{\nu+1}\backslash U_\nu)/G$ is Hausdorff.
\item[(b)] If $\nu$ is a limit ordinal, then $U_\nu=\bigcup_{\mu<\nu}U_\mu$.
\item[(c)] There exists an ordinal number $\nu_0$ such that $X=U_{\nu_0}$.
\end{itemize}
\end{itemize}
\end{thm}

\begin{defi}\label{defn-smooth}
If $X$ is a second countable locally compact almost Hausdorff $G$-space which satisfies the items (1),\ldots, (4) in Glimm's theorem,
 we say that the action of $G$ on $X$ is {\em smooth} (or $X$  is a {\em smooth} $G$-space).
 \end{defi}

We want to use Glimm's theorem to formulate a general orbit method for the Baum-Connes conjecture. For maximal flexibility we want 
to consider C*-algebras which can be thought of as section algebras of fields of C*-algebras over  an
almost Hausdorff space $X$ in the sense that there exists a continuous map $\varphi:\Prim(A)\to X$. 
In this situation, if $Y\subseteq X$ is any locally closed subset, we may identify the closed subset 
$\varphi^{-1}(\overline{Y})$ with the primitive ideal space of the quotient $A/I_Y$ with  $I_Y:=\cap\{P: P\in \varphi^{-1}(\overline{Y})\}$.
In a second step, we can identify the open subset $\varphi^{-1}(Y)$ of $\varphi^{-1}(\overline{Y})=\Prim(A/I_Y)$ with 
the primitive ideal space $\Prim(J_Y/I_Y)$ where $J_Y=\cap\{Q: Q\in \varphi^{-1}(\overline{Y}\smallsetminus Y)\}$. Hence 
for every locally closed subset $Y$ of $X$ we obtain a canonical subquotient $A_Y: =J_Y/I_Y$ of $A$ such that $\varphi^{-1}(Y)\cong \Prim(A_Y)$
In particular, since every one-point set $\{x\}$ of an almost Hausdorff space $X$ is locally closed, we obtain 
a canonical subquotient $A_x:=J_x/I_x$ such that $\varphi^{-1}(\{x\})\cong \Prim(A_x)$.
We call $A_x$ the {\em fibre of $A$ over $x$}.

We should note that if $X$ is Hausdorff, then it follows from the Dauns-Hoffmann theorem (see e.g. \cite[Theorem~A.34]{MR1634408}) that the continuous function $\varphi:\Prim(A)\to X$ 
induces a unique structure $\Phi:C_0(X)\to \mathcal{ZM}(A)$ of $A$ as a $C_0(X)$-algebra (here $\Phi$ is a nondegenerate $*$-homomorphism into the centre
$\mathcal{ZM}(A)$ of the multiplier algebra of $A$) such that 
$$I_x=\cap\{P: P\in \varphi^{-1}(\{x\})\}=\Phi(C_0(X\setminus\{x\}))A.$$
It follows that the fibre $A_x$ constructed above coincides with the usually defined fibre $A_x=A/(\Phi(C_0(X\smallsetminus\{x\}))A)$ of the 
$C_0(X)$-algebra $A$. Moreover, if $\varphi:\Prim(A)\to X$ is continuous {\bf and} open, then the corresponding $C_0(X)$-algebra is 
the section algebra of a continuous field of C*-algebras over $X$ with fibres $A_x$. For a reference see \cite[Theorem C.26]{MR2288954}.

We are now ready for the formulation of the general orbit method for BC with coefficients:

\begin{thm}[General orbit method for BC]\label{general-orbit}
Let $\alpha:G\to \Aut(A)$ be an action of  an exact second countable group $G$  on  a separable C*-algebra $A$.  Suppose 
that $X$ is a smooth $G$-space as in Definition \ref{defn-smooth} and  that $\varphi:\Prim(A)\to X$ is a continuous and open $G$-equivariant map. Then 
for each $x\in X$ the action $\alpha$ of $G$ on $A$ factors through a well-defined action $\alpha_x$ of the stabiliser $G_x$ on the fibre $A_x$ 
over $x$. Suppose further that the following two conditions are satisfied:
\begin{enumerate}
\item For each $x\in X$ the stabiliser $G_x$ satisfies BC for $A_x$.
\item  The sequence of  $G$-invariant open subsets $\{U_\nu\}_\nu$ of $X$ as in item (4) of Glimm's theorem
can be chosen such that $(U_{\nu+1}\smallsetminus U_\nu)/G$ is either totally disconnected or homeomorphic to the geometric realisation of a  finite-dimensional simplicial complex.
\end{enumerate}
Then $G$ satisifies BC for $A$.
\end{thm}

For the proof we need the following slight generalisation of \cite[Proposition 3.1]{MR2010742}.

\begin{prop}\label{prop-continuous-field}
Suppose that $\alpha:G\to \Aut(A)$ is an action of the second countable exact group on the separable C*-algebra $A$ 
and let $X$ be a second countable locally compact Hausdorff space equipped with the {\bf trivial $G$-action}. Suppose further that $\varphi:\Prim(A)\to X$ 
is  a continuous open $G$-invariant map such that the following two conditions are satisfied:
\begin{enumerate}
\item For all $x\in X$ the group $G$ satisfies BC for $A_x$.
\item $X$ is either totally disconnected or 
homeomorphic to the geometric realisation of a  finite-dimensional simplicial complex. 
\end{enumerate}
Then $G$ satisfies BC for $A$.
\end{prop}

\begin{rem}\label{rem-continuous-field}
It is really annoying that we are not able to remove the condition that  $X$ is totally disconnected or 
the geometric realisation of a finite dimensional simplicial complex. Although these assumptions 
are satisfied in many important situations, these assumptions result in the corresponding technical formulations 
for the asscending series of sets $\{U_\nu\}_{\nu}$ in the formulation of the orbit method for BC.
It would therefore be most desirable to give a proof of the above proposition without these assumptions!
\end{rem}

We should note that the statement  in Proposition~\ref{prop-continuous-field} implies that $A$ is the section algebra of a continuous field of C*-algebras over 
$X$ with fibres $A_x$ such that the action of $G$ on $A$ is $C_0(X)$-linear (i.e., $\alpha_g(\Phi(f)a)=\Phi(f)\alpha_g(a)$ for all $g\in G, f\in C_0(X)$ and 
$a\in A$). 

\begin{proof} [Proof of Proposition~\ref{prop-continuous-field}]
Note first that the only difference to the statement of \cite[Proposition 3.1]{MR2010742} is the fact, that we do not require 
that $G$ has a $\gamma$-element. To show that the proof given in \cite{MR2010742} still holds without the assumption of a $\gamma$-element, we 
need a substitute of \cite[Lemma 3.2]{MR2010742}. For this we use a recent result of \cite{BrCaLi16}, extending an earlier result by 
Ozawa \cite{MR1763912}, showing that a locally compact second countable group $G$ is exact if and only if it is {\em amenable at infinity}, which means that there exists a topologically amenable action of $G$ on a second countable compact Hausdorff space $Z$, say. 
It has been shown by Higson \cite{MR1779613} (see \cite[Theorem 1.9]{MR2100669} for the non-discrete case) that this implies
split injectivity for the Baum-Connes assembly map for every coefficient algebra $B$.

We need to recall the arguments for this fact:  let $Y$ be the space of 
probability measures on $Z$, we see that $Y$ is $K$-equivariantly contractible for any compact subgroup $K$ of $G$.
Then following the arguments of Higson in \cite{MR1779613} (see also \cite[Theorem 1.9]{MR2100669}), we see that for each $G$-C*-algebra $B$ 
we obtain a commutative diagram 
\begin{equation}\label{eq-diagram}
\begin{CD}
K_*^{\top}(G;B)    @>\mu_B >> K_*(B\rtimes_rG)\\
@ V(\iota_B)_* VV            @VV(\iota_B\rtimes_r G)_* V\\
K_*^{\top}(G; C(Y)\otimes B)    @>>\mu_{C(Y)\otimes B}>  K_*((C(Y)\otimes B)\rtimes_rG)
\end{CD}
\end{equation}
where $\iota_B: B\to C(Y)\otimes B$ is given by $\iota_B(b)=1_Y\otimes b$. In this diagram the left vertical map is an isomorphism by 
an application of the Going-Down principle \cite[Theorem 1.5]{MR2100669} (see the proof of \cite[Theorem 1.9]{MR2100669}). The
lower horizontal map is an isomorphism by a result of Tu (combine the main result of \cite{MR1703305} applied to the amenable groupoid 
$Y\rtimes G$ with \cite{MR1966758}). We therefore obtain a splitting homomorphism $\sigma_{B}: K_*(B\rtimes_rG)\to K_*^{\top}(G;B)$ for the assembly map $\mu_B$ given by 
$$\sigma_B=(\iota_B)_*^{-1}\circ \mu_{C(Y)\otimes B}^{-1}\circ (\iota_B\rtimes_rG)_*.$$
In particular, we see that 
 $\gamma_B: =\mu_B\circ \sigma_B\in \End(K_*(B\rtimes_rG))$ is an idempotent with $\gamma_B(K_*(B\rtimes_rG))=\mu_B(K_*^{\top}(G;B))$
 for all $B$ and it follows that $G$ satisfies the Baum-Connes conjecture for $B$ if and only if $(1-\gamma_B)\cdot K_*(B\rtimes_rG)=\{0\}$.
 
Now, since $G$ is exact, we know that given a short exact sequence of $G$-algebras $0\to J\to B\to B/J\to 0$ 
all maps in diagram (\ref{eq-diagram}) are natural with respect to the various six-term exact sequences in $K$-theory and topological $K$-theory 
which are attached to the  sequence  $0\to J\to B\to B/J\to 0$  (see \cite[Proposition 4.1]{MR1836047} for the case of the assembly maps). 
It follows from this that we obtain a short 
exact sequence of the form
$$\begin{CD}
(1-\gamma_J)\cdot K_0(J\rtimes _rG) @>>> (1-\gamma_B)\cdot K_0(B\rtimes _rG) @>>> (1-\gamma_{B/J})\cdot K_0(B/J\rtimes _rG)\\
@A\partial AA @.         @VV\partial V\\
(1-\gamma_{B/J})\cdot K_1(B/J\rtimes _rG) @<<< (1-\gamma_B)\cdot K_1(B\rtimes _rG) @<<< (1-\gamma_J)\cdot K_1(J\rtimes _rG).
\end{CD}
$$
This is now a complete analogue of \cite[Lemma 3.2]{MR2010742}, and the proof of the proposition follows now exactly as the proof 
of \cite[Proposition 3.1]{MR2010742}.
\end{proof}

\begin{proof}[Proof of Theorem \ref{general-orbit}]
We first need to verify that for all $x\in X$ the action of $G$ on $A$ factors through an action of $G_x$ on the fibre $A_x$.
To see this observe that $A_x=J_x/I_x$ with $J_x=\cap\{P: P\in \varphi^{-1}(\overline{\{x\}}\smallsetminus \{x\})\}$ and $I_x=
\cap\{P: P\in \varphi^{-1}(\overline{\{x\}})\}$. Since $\varphi$ is $G$-equivariant and $\overline{\{x\}}$ and $\{x\}$ are $G_x$-invariant 
subsets of $X$, both ideals are  $G_x$-invariant. But this implies that the action of $G_x$ on $A$ factors through a well-defined action on $A_x$.

Let $\{U_\nu\}_\nu$ be a sequence of $G$-invariant open sets as in the theorem and for each index $\nu$ 
let $X_\nu:=U_{\nu+1}\setminus U_\nu$. Then $X_\nu\subseteq X$ is a locally closed $G$-invariant subset of $X$
and therefore the action of $G$ on $A$ factors through an action $\alpha^\nu$ of $G$ on $A_\nu:=A_{X_\nu}$ for 
each ordinal $\nu$. We first show that $G$ satisfies BC for $A_\nu$ for all $\nu$.

For this recall that by construction of $A_\nu$ we have $\Prim(A_\nu)\cong \varphi^{-1}(X_\nu)$. Let $Y_\nu:=X_\nu/G$ and let 
$\psi:\Prim(A_\nu)\to Y_\nu$ denote the composition of $\varphi:\Prim(A_\nu)\to X_\nu$ with the quotient map $X_\nu\to Y_\nu$.
Then $\psi$ is a continuous  open $G$-invariant map and by the assumptions made in the theorem it follows from Proposition 
\ref{prop-continuous-field} that $G$ satisfies BC for $A_\nu$ if we can show that $G$ satisfies BC for $A_y$ for all $y\in Y_\nu$, where here 
$A_y$ is the fibre over $y\in Y_\nu$ with respect to the map $\psi$.
By construction we have $\Prim(A_y)=\psi^{-1}(\{y\})=\varphi^{-1}(G\cdot x)$ if $y=G\cdot x$ and therefore $A_y=A_{G\cdot x}$,
where $A_{G\cdot x}$ is the subquotient of $A$ corresponding to the locally closed subset $G\cdot x=\{gx: g\in G\}$ of $X$ with respect to 
$\varphi$. It follows from item (2) in Glimm's theorem that $G/G_x$  is $G$-homeomorphic to $G\cdot x$ via $gG_x\mapsto gx$.
Via the identification $G/G_x\cong G\cdot x$ we may view $\varphi$ as a continuous and $G$-equivariant map 
$\varphi:\Prim(A_y)\to G/G_x$. It follows then from \cite[Theorem]{MR994776} that $A_y$ is isomorphic to 
the induced algebra $\Ind_{G_x}^G A_x$. Since, by assumption, $G_x$  satisfies BC for $A_x$ 
it follows then from \cite[Theorem 2.5]{MR1836047} that $G$ satisfies BC for 
$A_y=\Ind_{G_x}^GA_x$. 

In order to complete the proof, let $A_{U_\nu}$ denote the ideal of $A$ corresponding to the open subset $U_\nu\subseteq X$.
We now want to show that $G$ satisfies BC for $A_{U_\nu}$ for all $\nu$. We do this by transfinite induction.
For $\nu=1$ we have $A_{U_1}=A_0$ (if we put $U_0:=\emptyset$, and hence $X_0=U_1$), and hence $G$ satisifies BC for $A_{U_1}$ 
by the arguments in the above paragraph. 
Assume now that we know that $G$ satisfies BC for $A_{U_\nu}$ for some $\nu$. 
We have a short exact sequence
$$0\to A_{U_\nu}\to A_{U_{\nu+1}}\to A_{\nu}\to 0$$
of $G$-C*-algebras. By the above paragraph we then know that $G$ satisfies BC for $A_{U_\nu}$ and for $A_\nu$.
Since $G$ is exact, it follows then from \cite[Proposition 4.2]{MR994776} that $G$ satisfies BC for $A_{U_{\nu+1}}$ as well.
Finally, if $\nu$ is a limit ordinal, then $U_\nu=\cup_{\mu<\nu} U_\mu$ and therefore 
$A_{{U_\nu}}=\overline{\cup_{\mu<\nu} A_{U_\mu}}=\lim_{\mu<\nu} A_{U_\nu}$. It follows then from 
\cite[Proposition 2.6]{MR2010742} that $G$ satisfies BC for $A_{U_\nu}$. Since $X=U_{\nu_0}$ for some ordinal number $\nu_0$,
we may now conclude that $G$ satisfies BC for $A$.
\end{proof}

The spectrum $\hat{A}$ of a $C^*$-algebra $A$ is the set of all unitary equivalence classes of irreducible representations of $A$. If $A$ is a separable type I $C^*$-algebra, then $\hat{A}$ equipped with the Jacobson topology is a locally compact second countable almost Hausdorff space and the map
$[\pi]\mapsto \ker\pi\in \Prim(A)$ is a homeomorphism between $\widehat{A}$ and $\Prim(A)$.  
Thus in case of type I coefficient algebras $A$ one could choose $X=\widehat{A}$ and $\varphi:\Prim(A)\to X$ the inverse of the above
homeomorphism. In this case the fibre $A_{[\pi]}$ for some class $[\pi]\in \widehat{A}$ is isomorphic to $\mathcal K(H_\pi)$, the algebra 
of compact operators on the Hilbert space $H_\pi$ of $\pi$. To see this recall that  by the type I condition we have $\mathcal K(H_\pi)\subseteq \pi(A)$. Let $J=\pi^{-1}(\mathcal K(H_\pi))$. Then one easily checks that $J=\cap\{\ker\sigma: \sigma\in \overline{\{\pi\}}\smallsetminus\{\pi\}\}=J_{\blue{[\pi]}}$ and 
hence $A_{[\pi]}=J_{[\pi]}/I_{[\pi]}=J/\ker\pi\cong \mathcal K(H_\pi)$. 
Hence, as a corollary of Theorem \ref{general-orbit} we get

\begin{cor}\label{cor-typeI}
Suppose that $G$ is an exact group acting on the separable  type I C*-algebra $A$ such that the action on $\widehat{A}$ is smooth.
Suppose further that there is an ascending sequence $\{U_\nu\}_{\nu}$ of open $G$-invariant subsets of $\widehat{A}$ as in Glimm's theorem
such that all difference sets $(U_{\nu+1}\smallsetminus U_\nu)/G$ are either totally disconnected or the geometric realisations of finite dimensional
simplicial complexes. Then $G$ satisfies BC for $A$ if each stabiliser $G_\pi$ satisfies BC for $\mathcal K(H_\pi)$.
 \end{cor}

If $G$ is a locally compact group, then we may identify its unitary dual $\hat{G}$ with the spectrum $\hat{C^*(G)}$ of the full group $C^*$-algebra of $G$ as topological spaces and we write $\widehat{G_r}$ for the subspace of $\widehat{G}$ corresponding to the quotient $C^*_r(G)$ of $C^*(G)$.
If $N$ is a closed normal subgroup of $G$, then we may decompose $C_r^*(G)$ as a reduced twisted crossed product 
$C_r^*(N)\rtimes_{(\alpha,\tau),r}(G,N)$ in the sense of Philip Green \cite{MR0493349}. It follows from \cite{MR1277761}
 that there exists an ordinary action of $\beta:G/N\to \Aut(B)$ on a C*-algebra $B$,
which is Morita equivalent to the decomposition twisted action $(\alpha,\tau)$  of $(G,N)$ on $C_r^*(N)$.

If $C_r^*(N)$ is type I and if $G_\pi\subseteq G$ denotes the 
stabiliser of $[\pi]\in \widehat{N_r}$ under the conjugation action, then $N\subseteq G_\pi$ and, similar to the case of ordinary actions as explained above, 
we obtain a twisted action of $(G_\pi, N)$ on $\mathcal K(H_\pi)$ for all $[\pi]\in \widehat{N_r}$. 
Since $(G,N)$-equivariant Morita equivalences induce $G/N$-equivariant homeomorphisms between the dual spaces
 and preserve the type I condition (e.g., see \cite[Propositions 5.11 and 5.12]{E-cross}), the given class 
 $[\pi]\in \widehat{N_r}$ corresponds to a class $[\tilde\pi]\in \widehat{B}$ such that $G_\pi/N=(G/N)_{\tilde\pi}$ and the corresponding action of 
 $(G/N)_{\tilde\pi}$ on $\mathcal K(H_{\tilde\pi})$ is Morita equivalent to the twisted action of $(G_\pi, N)$ on $\mathcal K(H_\pi)$.
 In \cite{MR1857079} the authors extended the the Baum-Connes conjecture to the case of twisted actions and showed 
 in particular that the conjecture is invariant under Morita equivalence of twisted actions. 
 Therefore $(G,N)$ satisfies BC for $C_r^*(N)$ if and only if $G/N$ satisfies BC for $B$ and $(G_\pi, N)$ satisfies (twisted) BC for 
 $\mathcal K(H_\pi)$ if and only if $(G/N)_{\tilde\pi}$ satisfies BC for $\mathcal K(H_{\tilde\pi})$.  Using this, we now get
 
 \begin{cor}[BC for group extensions]\label{orbit BC}
Suppose that $N$ has the Haagerup property and is a closed normal subgroup of a second countable locally compact group $G$ such that the following conditions are satisfied:
\begin{enumerate}
\item $G/N$ is an exact group.
\item $C_r^*(N)$ is type I and the action of $G/N$ on $\widehat{N_r}$ is smooth.
\item For all $\pi\in \widehat{N}_r$ the pair $(G_\pi, N)$ satisfies (twisted) BC for $\mathcal K(H_\pi)$.
\item There exists a  sequence of $G$-invariant open subsets $\{U_\nu\}_\nu$ of $\widehat{N}_r$ as in item (4) of Glimm's theorem such 
that $(U_{\nu+1}\smallsetminus U_\nu)/G$ is either totally disconnected or homeomorphic to the geometric realisation of a finite simplicial complex .
\end{enumerate}
Then $G$ satisfies the Baum-Connes conjecture for $\C$.
\end{cor}
\begin{proof}
Since $N$ is assumed to have the Haagerup property, it follows from \cite[Theorem 2.1]{MR2100669} that $G$ satisfies BC for $\C$ if and only if 
$(G,N)$ satisfies BC for $C_r^*(N)$ with respect to the twisted decomposition action $(\alpha,\tau)$. 
Let $\beta:G/N\to \Aut(B)$ be an ordinary action of $G/N$ which is Morita equivalent to the twisted action $(\alpha,\tau)$.
Then all assumptions of the corollary carry over to the action $\beta$ -- in particular we have that each stabiliser $(G/N)_{\tilde\pi}$ 
satisfies BC for $\mathcal K(H_{\tilde\pi})$ for all $\tilde\pi\in \widehat{B}$. Thus it follows from 
Corollary \ref{cor-typeI} that $G/N$ satisfies BC for $B$. But by Morita equivalence we  then conclude that $(G,N)$ satisfies BC 
for $C_r^*(N)$.
\end{proof}

\begin{rem}\label{rem-semidirect}
If in the above corollary the group $G$ is a semidirect product $N\rtimes G/N$ (i.e., the extension $1\to N\to G\to G/N\to 1$ is split)
then there exists an ordinary action $\beta: G/N\to\Aut(C_r^*(N))$ such that $C_r^*(G)\cong C_r^*(N)\rtimes_{\beta,r} G/N$.
In this case it is not necessary to use twisted actions and we can directly work with the action of $G/N$ on $C_r^*(N)$ in the  
above corollary. This will be the situation in most of our examples below.
\end{rem}

\section{BC for certain algebraic groups over local function fields}

In this section, we will use the orbit method for group extensions to verify the Baum-Connes conjecture for groups of $k$-rational points of some Levi-decomposable linear algebraic groups over a local function field $k$. 

We begin with a recapitulation of Kirillov's orbit method for unipotent groups $N$, which describes the unitary dual $\widehat{N}$ in terms of 
the co-adjoint orbits in $\mathfrak{n}^*$, where $\mathfrak{n}$ denotes a suitable version of a Lie algebra for $N$. Let $k$ be a (non-archimedean) local field of positive characteristic $p$ and let $\mathfrak{n}$ be any finite-dimensional nilpotent Lie algebra over $k$ with nilpotence length $l<p$. It is well-known from \cite[Theorem~1]{MR0194482} that $\mathfrak{n}$ admits a faithful finite-dimensional linear representation by nilpotent matrices, which we can simultaneously triangularize. Hence, the Campbell-Hausdorff formula implies that $N:=\exp(\mathfrak{n})$ is a linear algebraic unipotent group defined over $k$ and its Lie algebra is $\mathfrak{n}$. Here we write $N$ for its $k$-rational points $N(k)$ for simplicity.

The adjoint action $\operatorname{Ad}:N\ra \text{GL}(\mathfrak{n})$ is defined by
\begin{align*}
\operatorname{Ad}(n)(X):=\log(n\exp(X)n^{-1})\ \text{for $n\in N$ and $X\in \mathfrak{n}$},
\end{align*}
where $\log: N\ra \mathfrak{n}$ denotes the inverse of $\exp$.
Let $\mathfrak{n}^*$ denote the space of $k$-linear functionals of $\mathfrak{n}$ and let $\operatorname{Ad}^*:N\ra \text{GL}(\mathfrak{n}^*)$ be the co-adjoint action given by $\operatorname{Ad}^*(n)(f)=f\circ \operatorname{Ad}(n^{-1})$ for $n\in N$ and $f\in \mathfrak{n}^*$, then $\operatorname{Ad}^*$ is a $k$-rational action as $p>l$. Hence, all $\operatorname{Ad}^*(N)$-orbits are locally closed in the Hausdorff topology of $\mathfrak{n}^*$ (see \cite[Theorem~6.15 and Proposition~6.8 c)]{MR0425030} or Theorem~\ref{locally closed orbits}). It turns out that the quotient space $\mathfrak{n}^*/\operatorname{Ad}^*(N)$ is homeomorphic to the unitary dual $\hat{N}$ of $N$.

For a description of this homeomorphism, let $\Lambda_l$ denote the ring $\Z[\frac{1}{l!}]$ for $l$ the nilpotence length of $\mathfrak{n}$. Then $\mathfrak{n}$ becomes a Lie-algebra over the ring $\Lambda_l$
as considered in \cite{MR3012147}.
Now we fix a character $\varepsilon\in \hat{k}$ of order zero (see \cite[II. Definition~4]{MR0427267}). For $f\in \mathfrak{n}^*$, let $\mathfrak{m}_f$ be a maximal closed $\Lambda_l$-Lie subring of $\mathfrak{n}$ which is subordinate to $f$ in the sense that $\epsilon\circ f([\mathfrak{m}_f,\mathfrak{m}_f])=\{0\}$, where $[\cdot ,\cdot]$ denotes the Lie bracket. If $M_f:=\exp(\mathfrak{m}_f)\subseteq N$, then the map $m\mapsto \chi_f(m):=\varepsilon(f(\log(m)))$ defines a character of $M_f$. 
The following theorem then follows from \cite[Theorem~II]{MR0492059} or \cite[Theorem 6.14, Theorem 7.1 and Example 8.3]{MR3012147} (observe that 
in this situation we have $\mathfrak n^*\cong\widehat{\mathfrak n}$ via $f\mapsto \varepsilon\circ f$ by \cite[II. Theorem~3]{MR0427267}). 

\begin{thm}[cf. \cite{MR0492059, MR3012147}]\label{orbit for unipotent}
In the above situation the  
 induced representation $\pi_f:=\text{Ind}_{M_f}^N \chi_f$ is an irreducible unitary representation of $N$ and its equivalence class does not depend on the choice of $\mathfrak{m}_f$. 
The resulting map $\mathfrak{n}^*\ra \hat{N}$ given by $f\mapsto \pi_f$, is constant on $\operatorname{Ad}^*(N)$-orbits and it induces a homeomorphism between $\mathfrak{n}^*/\operatorname{Ad}^*(N)$ and $\hat{N}$. 
Moreover, the group $C^*$-algebra $C^*(N)$ is type $I$.
\end{thm}

\begin{rem}\label{normal orbit}
Let $N=\exp{\mathfrak n}$ as in the above discussion. Let $R$ be any locally compact (not necessarily algebraic) group acting continuously on $N$ 
and let $G:=N\rtimes R$. As before, we define the adjoint action $\operatorname{Ad}$ of $G$ on $\mathfrak{n}$ by the formula
\begin{align*}
\operatorname{Ad}(g)(X):=\log(g\exp(X)g^{-1})\ \text{for $g\in G$ and $X\in \mathfrak{n}$}.
\end{align*}
Let $\operatorname{Ad}^*:G\ra \text{GL}(\mathfrak{n}^*)$ be its co-adjoint action given by $\operatorname{Ad}^*(g)(f)=f\circ \operatorname{Ad}(g^{-1})$ for $g\in G$ and $f\in \mathfrak{n}^*$.
Let $\mathfrak{m}_f$ be a maximal closed $\Lambda_l$-Lie subring of $\mathfrak n$ subordinate to $f$.
 It is straightforward to see that $\operatorname{Ad}(g)(\mathfrak{m}_f)$ is a 
maximal $\Lambda_l$-Lie subring of $\mathfrak n$ subordinate to $\operatorname{Ad}^*(g)(f)$, hence we may choose
$\mathfrak{m}_{\operatorname{Ad}^*(g)(f)}=\operatorname{Ad}(g)(\mathfrak{m}_f)$ for $g\in G$ and $f\in \mathfrak{n}^*$.  Hence,
\begin{align*}
M_{\operatorname{Ad}^*(g)(f)}=\exp(\mathfrak{m}_{\operatorname{Ad}^*(g)(f)})=\exp(\operatorname{Ad}(g)(\mathfrak{m}_f))=gM_fg^{-1}.
\end{align*}
Let $\chi^g_f:gM_fg^{-1}\ra \mathbb{T}$ be the character of $gM_fg^{-1}$ given by $\chi^g_f(gm g^{-1}):=\chi_f(m)$, $m\in M_f$. Then
a short computation shows that $\chi_f^g=\chi_{\operatorname{Ad}^*(g)(f)}$, which implies that
\begin{align*}
\pi_{\operatorname{Ad}^*(g)(f)}:=\text{Ind}^N_{M_{\operatorname{Ad}^*(g)(f)}}\chi_{\operatorname{Ad}^*(g)(f)}=\text{Ind}^N_{gM_fg^{-1}} \chi_f^g\cong (\text{Ind}^N_{M_f} \chi_f)^g=\pi_f^g,
\end{align*}
where $\pi_f^g(n):=\pi_f(g^{-1}ng)$ denotes the $g$-conjugate of $\pi_f$ in $\widehat{N}$ (see e.g. \cite[Proposition~1.12]{MR1066810}).
If $g=(n,r)\in N\rtimes R$ and $f\in \mathfrak{n}^*$, then 
\begin{align*}
\pi_{\operatorname{Ad}^*(g)(f)}\cong \pi_f^g=(\pi_f^r)^n\cong \pi_f^r.
\end{align*}
It follows that the Kirillov-orbit homeomorphism $\mathfrak{n}^*/\operatorname{Ad}^*(N) \ra \hat{N}$ is $\operatorname{Ad}^*(R)-R$ equivariant if we 
define the action of $R$ on $\widehat{N}$ by $r\cdot \pi:=\pi^r$. In particular, the Kirillov map  induces a homeomorphism $\mathfrak{n}^*/\operatorname{Ad}^*(G)=\left(\mathfrak{n}^*/\operatorname{Ad}^*(N)\right)/ \operatorname{Ad}^*(R) \cong \hat{N}/R.$
\end{rem}

 \begin{rem}
Let $k$ be a local field of positive characteristic $p$. Then we have $(x+y)^p=x^p+y^p$ for all $x,y\in k$. The fake Heisenberg group is defined as follows:
$$
N:=\left\{\begin{pmatrix}
1&a&b\\
0&1&a^p\\
0&0&1\\
\end{pmatrix}: a,b\in k\right\}_.$$
It is a non-abelian $k$-split connected unipotent linear algebraic group with nilpotence length two. Consider its Lie ring
$$
\mathfrak{n}:=\log(N)=\left\{\begin{pmatrix}
0&x&y\\
0&0&x^p\\
0&0&0\\
\end{pmatrix}: x,y\in k\right\}_.$$
Since $k$ is not perfect, $\mathfrak{n}$ is not stable under scalar multiplication. In particular, $\mathfrak{n}$ is not a vector subspace of $M_3(k)$ and is different from Lie($N$), the Lie algebra associated to $N$ in the sense of algebraic geometry. Hence, $\mathfrak{n}^*$ does not make sense as set of linear functionals 
from $\mathfrak n$ to $k$. But the orbit method still makes sense if we replace $\mathfrak{n}^*$ by the Pontryagin dual $\hat{\mathfrak{n}}$ of the additive group $\mathfrak{n}$ (see \cite[Example~8.2 and Example~8.3]{MR3012147}). However, we do not know whether all $\operatorname{Ad}^*(N)$-orbits are locally closed in the Hausdorff topology of $\hat{\mathfrak{n}}$. 
\end{rem}

\begin{cor}\label{orbit for BC}
Let $k$ be a non-archimedean local field of positive characteristic $p$ and let $\mathfrak{n}$ be a finite-dimensional nilpotent $k$-Lie algebra with nilpotence length $l<p$. Let $N:=\exp(\mathfrak{n})$ be the corresponding unipotent $k$-group and let $R$ be a locally compact second countable exact group. 
Suppose that $R$ acts continuously on $N$ such that all $R$-orbits in $\hat{N}$ are locally closed
and  that for all $[\pi]\in \hat{N}$ the stabilizers  $R_\pi=\{ g\in R: g\cdot[\pi]=[\pi]\}$ satisfy BC  for $\mathcal{K}(H_\pi)$.

Then the semidirect product group $G=N\rtimes R$ satisfies BC for $\C$.
\end{cor}
\begin{proof}
Since nilpotent groups are amenable, we have $C^*(N)=C_r^*(N)$.
According to Theorem~\ref{orbit for unipotent}, 
$C^*(N)$ is type I and the Kirillov-orbit map $\mathfrak{n}^*/\operatorname{Ad}^*(N) \ra \widehat{N}$ is a homeomorphism.
Since all $R$-orbits in $\hat{N}$ are locally closed, it follows from Glimm's Theorem that there exists a sequence of $G$-invariant open subsets $\{U_\nu\}_\nu$ of $\hat{N}$ such that $(U_{\nu+1}\backslash U_\nu)/G$ is Hausdorff. By Corollary~\ref{orbit BC} it suffices to show that the Hausdorff space $X_\nu:=(U_{\nu+1}\backslash U_\nu)/G$ is totally disconnected for every $\nu$. 
 By Remark~\ref{normal orbit} we may identify $\hat{N}/G\cong \widehat{N}/R$ with $\mathfrak{n}^*/\operatorname{Ad}^*(G)$ as topological spaces. If $p:\mathfrak{n}^*\ra \mathfrak{n}^*/\operatorname{Ad}^*(G)$ denotes the orbit map, then $p^{-1}(X_\nu)$ is locally closed in the locally compact Hausdorff totally disconnected space $\mathfrak{n}^*$. Hence, $p^{-1}(X_\nu)$ is a locally compact Hausdorff totally disconnected $\operatorname{Ad}^*(G)$-invariant space. Hence, $p$ restricts to the orbit map from $p^{-1}(X_\nu)$ onto $X_\nu=p^{-1}(X_\nu)/\operatorname{Ad}^*(G)$. We conclude that $X_\nu$ is totally disconnected.
\end{proof}

Recall that the group $G=\mathbf{G}(k)$ of $k$-rational points of a linear algebraic group $\mathbf G$ is isomorphic to a Zariski closed (in particular, Hausdorff closed) subgroup of $\GL_n(k)$ for some $n\in \N$. The \emph{unipotent radical} $\mathbf N$ of $\mathbf G$ is the largest connected unipotent normal subgroup of $\mathbf{G}(\bar{k})$, where $\bar{k}$ is an algebraic closure of $k$. A linear algebraic group is called \emph{reductive} if its unipotent radical is trivial. One says that a linear algebraic group $\mathbf G$ defined over a field $k$ has a \emph{Levi decomposition over $k$} if its unipotent radical $\mathbf N$ is defined over $k$ and there exists a Zariski closed reductive $k$-subgroup $\mathbf R$, called a Levi factor, of $\mathbf G$ such that the product mapping $\mathbf N\rtimes \mathbf R \ra \mathbf G$ given by $(x,y)\mapsto xy$ is a $k$-isomorphism of algebraic groups. In particular, $G\cong N\rtimes R$ if
$N=\mathbf{N}(k)$ and $R=\mathbf R(k)$ denote the $k$-rational points of $\mathbf N$ and $\mathbf R$, respectively. It is well-known that linear algebraic groups in characteristic zero always have a Levi decomposition (see e.g. \cite[VIII, Theorem~4.3]{MR620024}). When the field has positive characteristic, $\mathbf G$ need not have a Levi factor and the unipotent radical of $\mathbf G$ is in general not defined over the field (see \cite{MR0218365, MR2753264, MR3362817} for more details about Levi decompositions over fields of positive characteristic). Therefore, we will only consider such linear algebraic groups which have a Levi decomposition over a local function field.

The orbit method for the Baum-Connes conjecture for the group $G=\mathbf G(k)$ of $k$-rational points of 
a linear algebraic group $\mathbf G$ over a local function field $k$ can be simplified: since $\GL_n(k)$ admits a solvable cocompact closed subgroup, it follows from \cite[Theorem~7]{MR1725812} that $\GL_n(k)$ is an exact group. Hence, the closed subgroup $G\subseteq \GL_n(k)$ is also exact by \cite[Theorem~4.1]{MR1725812}. 
Moreover, we have the following well-known theorem about locally closed orbits:

\begin{thm}[{\cite[Theorem~6.15 and Proposition~6.8 c)]{MR0425030}}]\label{locally closed orbits}
Let $k$ be a non-archimedean local field and $\mathbf X$ be an algebraic $k$-variety. If $\mathbf G$ is a linear algebraic $k$-group, which acts $k$-rationally on $\mathbf X$, then the induced action of $G=\mathbf G(k)$ on $X=\mathbf X(k)$ is continuous and all $G$-orbits in $X$ are locally closed.
\end{thm}
The following theorem is now almost an immediate consequence of Corollary~\ref{orbit for BC}:
\begin{thm}\label{algebraic orbit method}
Let $\mathbf G$ be a linear algebraic group with a Levi decomposition $\mathbf G=\mathbf N\rtimes \mathbf R$ over a non-archimedean local field $k$ of positive characteristic $p$. Let $N=\mathbf N(k)$, $R=\mathbf R(k)$ and $G=N\rtimes R$ ($=\mathbf G(k)$).
Assume that $N=\exp(\mathfrak{n})$ for a finite-dimensional nilpotent $k$-Lie algebra $\mathfrak{n}$ with nilpotence length $l<p$.

Then $G=N\rtimes R$ satisfies BC for $\C$, if all stabilizers $R_\pi$  satisfy BC for $\mathcal{K}(H_\pi)$ for all $[\pi]\in \hat{N}$.
\end{thm}
\begin{proof}
By Corollary~\ref{orbit for BC} it suffices to show that $R$ acts continuously on $N$ and all $R$-orbits in $\hat{N}$ are locally closed. The continuity of the action follows directly from Theorem~\ref{locally closed orbits}. By Glimm's Theorem and Remark~\ref{normal orbit}, we only have to show that $\hat{N}/R\cong \mathfrak{n}^*/\operatorname{Ad}^*(G)$ is almost Hausdorff. Using Glimm's theorem again, this follows if the $\operatorname{Ad}^*(G)$-orbits in 
$\mathfrak n^*$ are locally closed. 
 Since the co-adjoint action $\operatorname{Ad}^*$ of $G$ on $\mathfrak{n}^*$ is $k$-rational when $l<p$, this follows from Theorem~\ref{locally closed orbits}.
\end{proof}
In order to give some examples for which  the above theorem  applies, we use the following proposition, which is a slight extension of \cite[Proposition~4.9]{MR2010742}.
\begin{prop}\label{reductive}
Let $k$ be a non-archimedean local field and $G=\mathbf{G}(k)$ be the $k$-rational points of a reductive linear algebraic group $\mathbf G$. If $1\ra C\ra \tilde{G}\ra G\ra 1$ is a (not necessarily algebraic) topological group extension with a compact second countable normal subgroup $C$, then $\tilde{G}$ satisfies BC for the compact operators $\mathcal{K}$ on all separable Hilbert spaces and with respect to any actions.
\end{prop}
\begin{proof}
It is well-known that $G$ acts continuously and properly isometrically on its affine Bruhat-Tits building $\mathcal{B}(G)$ (see e.g.\cite[Section~2]{MR546588}). Since $C$ is compact, 
it follows that the action of $G$ on $\mathcal{B}(G)$ inflates to a proper action of  $\tilde{G}$. It is shown in \cite[Section~4.3]{MR1914617} that $G$ has the property (HC) in the sense of \cite[Section~4.1]{MR1914617}. It follows then from \cite[Lemma~4.3]{MR2010742} that $\tilde{G}$ also has the property (HC).

In order to show that $\tilde{G}$ satisfies BC for $\mathcal{K}$ with respect to any actions, it suffices to show that every central extension $\bar{G}$ of $\tilde{G}$ by $\mathbb{T}$ satisfies BC for $\C$ by \cite[Proposition~2.7]{MR2010742}. By the same arguments as before, we conclude that $\bar{G}$ also acts continuously properly isometrically on $\mathcal{B}(G)$ and $\bar{G}$ has the property (HC). Hence, $\bar{G}$ satisfies BC for $\C$ by \cite[Corollary~0.0.3 and Proposition~4.1.2]{MR1914617}.
\end{proof}
%
Proposition~\ref{reductive} implies that the Baum-Connes conjecture holds for all Levi-decomposable linear algebraic groups with a trivial unipotent radical, which has nilpotence length zero. In the following corollary we deal with vector groups, which have nilpotence length one.

\begin{cor}\label{ex1}
Let $k$ be a non-archimedean local field of positive characteristic $p$ and let  $n\in \N$. Then the following groups satisfy BC for $\C$:
\begin{itemize}
\item[(1)] $k^n\rtimes \GL_n(k)$, where $\GL_n(k)$ acts on $k^n$ by matrix multiplications.
\item[(2)] $k\rtimes \GL_n(k)$, where $\GL_n(k)$ acts on $k$ by $g.v=\det(g)^p \cdot v$ for $g\in \GL_n(k)$ and $v\in k$.
\end{itemize}
\end{cor}
\begin{proof}
Since $\GL_n(k)$ acts on $k^d$ by $k$-linear endomorphism for $d\in \{1,n\}$, it follows that $\hat{k^d}$ is $\GL_n(k)$-equivariantly topologically isomorphic to $k^d$ with the action of $\GL_n(k)$ given by $(g,v)\mapsto (g^{-1})^t\cdot v$ for $g\in \GL_n(k)$ and $v\in k^d$. By Theorem~\ref{algebraic orbit method} it suffices to show that the stabilizer $\GL_n(k)_v$ satisfies BC for $\C$ for all $v\in k^d$.

(1): Since $\GL_n(k)$ acts transitively on $k^n\backslash\{0\}$, there are only two orbits: $\{0\}$ and $k^n\backslash\{0\}$. Then the following are all stabilizers up to conjugacy: $\GL_n(k)_{\{0\}}=\GL_n(k)$ and $\GL_n(k)_{e_1}=k^{n-1}\rtimes \GL_{n-1}(k)$, where $e_1=(1,0,\ldots,0)^t\in k^n$. Since $\GL_n(k)$ is reductive, it satisfies BC for $\C$ by Proposition~\ref{reductive}. Moreover, $k^{n-1}\rtimes \GL_{n-1}(k)$ satisfies BC for $\C$ by an induction argument as long as the conjecture holds for $k\rtimes \GL_{1}(k)$. Since $k\rtimes \GL_{1}(k)$ is amenable, we complete the proof by \cite{MR1821144}.

(2): We have to show that $\GL_n(k)_v$ satisfies BC for $\C$ for every $v\in k$. It is clear that $\GL_n(k)_{\{0\}}=\GL_n(k)$. Let $v\in k\backslash \{0\}$. Since $k$ has characteristic $p$, it follows that
\begin{align*}
\GL_n(k)_v&=\{g\in \GL_n(k): \det(g)^p-1=0\}\\
          &=\{g\in \GL_n(k): (\det(g)-1)^p=0\}\\
          &=\{g\in \GL_n(k): \det(g)-1=0\}\\
          &=\SL_n(k).
\end{align*}
Since both $\GL_n(k)$ and $\SL_n(k)$ are reductive, we are done by Proposition~\ref{reductive}.
\end{proof}
\begin{rem}
By a proof almost identical to the proof of Corollary~\ref{ex1} (1) we see that $k^n\rtimes \SL_n(k)$ satisfies BC for $\C$ for a local (function) field $k$. This is already known from \cite[Example~4.3]{MR1836047}.
\end{rem}

We end this paper with some more advanced examples. One of these is the Jacobi group, which is the semidirect product of the symplectic group and the Heisenberg group. Its unipotent radical is the Heisenberg group, which has nilpotence length two.

Let $k$ be a non-archimedean local field. We will assume that $k$ is \emph{not} of characteristic 2. Recall that the \emph{symplectic group} $\Sp_{2n}(k)$ is the closed subgroup of $\GL_{2n}(k)$ consisting of all matrices $g$ with $g^tJg=J$, where 
\begin{align*}
J=\begin{pmatrix} 0 & I_n  \\ -I_n & 0 \end{pmatrix}
\end{align*}
and $I_n$ is the $n\times n$ identity matrix. Observe that $\Sp_2(k)=\SL_2(k)$. The symplectic group $\Sp_{2n}(k)$ has a unique perfect two-fold central extension $\Mp_{2n}(k)$ which is called the \emph{metaplectic group} (see \cite{MR0165033, MR777342}): 
$$1\ra \Z_2\ra \Mp_{2n}(k) \ra \Sp_{2n}(k) \ra 1.$$
However, $\Mp_{2n}(k)$ is not a linear group. Since exactness is stable under extensions (see \cite[Theorem~5.1]{MR1725812}), we see that $\Mp_{2n}(k)$ is an exact group. 
Moreover, $\Mp_{2n}(k)$ satisfies BC for the compact operators $\mathcal{K}$ with respect to any actions by Proposition~\ref{reductive}. Let $\Sp_{2n}(k)$ act on $k^{2n}$ by matrix multiplications and let $\Mp_{2n}(k)$ act on $k^{2n}$ through the action of $\Sp_{2n}(k)$.

Now consider the symplectic form $\omega$ on $k^{2n}$ given by $$\omega(x,y)=x^tJy,\quad \text{for $x,y\in k^{2n}$}.$$
The $(2n+1)$-dimensional \emph{Heisenberg group} over $k$ is the group $\HH_{2n+1}(k)$ with underlying set $k^{2n}\times k$ and group multiplication  given by
\begin{align*}
(x,\lambda)(y,\mu)=(x+y,\lambda+\mu+\frac{\omega(x,y)}{2}),\quad \text{for $x,y\in k^{2n}$ and $\lambda,\mu\in k$}.
\end{align*}
The symplectic group $\Sp_{2n}(k)$ acts on $\HH_{2n+1}(k)$ by automorphisms:
\begin{align*}
g\cdot (x,\lambda)=(g x,\lambda)\quad \text{for $g\in \Sp_{2n}(k)$, $x\in k^{2n}$ and $\lambda\in k$}. 
\end{align*}
Let $\Mp_{2n}(k)$ act on $\HH_{2n+1}(k)$ through the action of $\Sp_{2n}(k)$.
\begin{cor}\label{ex2}
Let $k$ be a non-archimedean local field of positive characteristic $p>2$ and $n\in \N$. Then the following groups satisfy BC for $\C$:
\begin{itemize}
\item[(1)] $\HH_{2n+1}(k)\rtimes \Sp_{2n}(k)$,
\item[(2)] $k^{2n}\rtimes \Sp_{2n}(k)$,
\item[(3)] $\HH_{2n+1}(k)\rtimes \Mp_{2n}(k)$,
\item[(4)] $k^{2n}\rtimes \Mp_{2n}(k)$.
\end{itemize}
\end{cor}
\begin{proof}
(1): It follows from Theorem~\ref{algebraic orbit method} that it suffices to show that $\Sp_{2n}(k)_\pi$ satisfies BC for $\mathcal{K}(H_\pi)$ for all $[\pi]\in \hat{\HH_{2n+1}(k)}$. It is well-known from the Stone-von Neumann theorem that $\hat{\HH_{2n+1}(k)}=\{[\chi_{v,w}]\}_{v,w\in k^n}\cup \{[\pi_\lambda]\}_{\lambda\in k\backslash \{0\}}$, where $\chi_{v,w}$ is a unitary character and $\pi_\lambda$ is an infinite-dimensional irreducible unitary representation of $\HH_{2n+1}(k)$ on $L^2(k^n)$. If $g\in \Sp_{2n}(k)$, then the action is given by
\begin{align*}
&g\cdot \chi_{v,w}=\chi_{(g^{-1})^t\cdot (v,w)^t}\quad \text{for all $v,w\in k^n$} &\text{and}& &g\cdot \pi_\lambda\cong \pi_\lambda \quad \text{for all $\lambda\in k\backslash \{0\}$}.
\end{align*}
We refer to \cite[Section~1]{MR777342} and \cite[Theorem~5.2 and Corollary~5.3]{MR1486137} for more details about the Stone-von Neumann theorem.

Since $\Sp_{2n}(k)_{\pi_\lambda}=\Sp_{2n}(k)$ is reductive, it satisfies BC for $\mathcal{K}(H_{\pi_\lambda})$ by Proposition~\ref{reductive}. Since $\{[\chi_{v,w}]\}_{v,w\in k^n}\cong \hat{k^{2n}}$ is  $\Sp_{2n}(k)$-equivariantly topologically isomorphic to $k^{2n}$, it is equivalent to show that $\Sp_{2n}(k)_{{(v,w)}^t}$ satisfies BC for $\C$ for all $v,w\in k^n$.

Since $\Sp_{2n}(k)$ is stable under transpositions, it acts transitively on $k^{2n}\backslash\{0\}$. Hence, the following are all stabilizers up to conjugacy: $\Sp_{2n}(k)_{\{0\}}=\Sp_{2n}(k)$ and $\Sp_{2n}(k)_{e_1}=\HH_{2(n-1)+1}(k)\rtimes \Sp_{2(n-1)}(k)$ for $n\geq 2$, where $e_1=(1,0,\ldots,0)^t\in k^{2n}$. Since $\Sp_{2n}(k)$ is reductive, it satisfies BC for $\C$ by Proposition~\ref{reductive}. Moreover, $\HH_{2(n-1)+1}(k)\rtimes \Sp_{2(n-1)}(k)$ satisfies BC for $\C$ by an induction argument as long as the conjecture holds for $\HH_{3}(k)\rtimes \Sp_{2}(k)$. Since $\Sp_{2}(k)=\SL_2(k)$ has the Haagerup property, it follows from \cite[Theorem~2.1]{MR2100669} and \cite{MR1821144} that $\HH_{3}(k)\rtimes \Sp_{2}(k)$ satisfies BC for $\C$. Therefore, we complete the proof.

(2): Since $\Sp_{2}(k)=\SL_2(k)$ has the Haagerup property, it follows from \cite[Theorem~2.1]{MR2100669} and \cite{MR1821144} that $k^2\rtimes \Sp_{2}(k)$ satisfies BC for $\C$. We may assume that $n\geq 2$.  It follows from Theorem~\ref{algebraic orbit method} that it suffices to show that $\Sp_{2n}(k)_v$ satisfies BC for $\C$ for all $v\in k^{2n}$ with the action of $\Sp_{2n}(k)$ given by $(g,v)\mapsto (g^{-1})^t\cdot v$ for $g\in \Sp_{2n}(k)$ and $v\in k^{2n}$. Since $\Sp_{2n}(k)$ is stable under transpositions, it acts transitively on $k^{2n}\backslash\{0\}$. Hence, the following are all stabilizers up to conjugacy: $\Sp_{2n}(k)_{\{0\}}=\Sp_{2n}(k)$ and $\Sp_{2n}(k)_{e_1}=\HH_{2(n-1)+1}(k)\rtimes \Sp_{2(n-1)}(k)$ for $n\geq 2$. All of them satisfy BC for $\C$ as we have already shown in (1). Therefore, we complete the proof of (2).

Since $\Mp_{2n}(k)$ is not a linear group, we have to apply Corollary~\ref{orbit for BC} for (3) and (4). The metaplectic group $\Mp_{2n}(k)$ is a locally compact second countable exact group. Moreover, it acts continuously on $\HH_{2n+1}(k)$ and $k^{2n}$ such that all $\Mp_{2n}(k)$-orbits in $\hat{\HH_{2n+1}(k)}$ and $\hat{k^{2n}}$ are locally closed as it acts through the action of $\Sp_{2n}(k)$.

The proofs for (3) and (4) are almost identical to the proofs for (1) and (2) as long as we notice the following facts. $\Mp_{2n}(k)_{\pi_\lambda}=\Mp_{2n}(k)$ satisfies BC for the compact operators $\mathcal{K}$ with respect to any actions and $\Mp_{2n}(k)_{e_1}=\HH_{2(n-1)+1}(k)\rtimes \Mp_{2(n-1)}(k)$ for $n\geq 2$. Finally, we notice that $\Mp_{2}(k)$ has the Haagerup property by \cite[Proposition~2.5 (2)]{MR1756981}. We leave the details to the reader.
\end{proof}

\begin{rem} We hope that the methods described in this paper will eventually help to find a proof of the Baum-Connes conjecture
with trivial coefficient for all (groups of $k$-rational points of) linear algebraic groups over a local field $k$ with positive characteristic. 
The case of local fields with zero characteristic has been solved in \cite{MR2010742} using similar methods.
\end{rem}

\bibliographystyle{plain}
\bibliography{kangbib}
\end{document}